\documentclass[12pt]{amsart}
\pdfoutput=1

\usepackage{enumerate}
\usepackage{amsmath,amsthm,amssymb}
\usepackage{color}
\usepackage{url}

\usepackage{graphicx}
\usepackage{tikz}
\usetikzlibrary{calc}
\usepackage{wrapfig}

\newcommand{\projectivespace}[0]{\mathbb{P}^N}

\newcommand{\variety}[0]{X}
\newcommand{\degree}[0]{d_X}
\newcommand{\dualvariety}[0]{X^{\vee}}
\newcommand{\segreembedding}[0]{X\times \mathbb{P}^{n-1}}
\newcommand{\hyperdiscriminant}[0]{\Delta_{X\times \mathbb{P}^{n-1}}}
\newcommand{\hyperdiscriminantsurface}[0]{\Delta_{X\times \mathbb{P}^{1}}}
\newcommand{\toricvariety}[0]{X_A}
\newcommand{\conv}[0]{\mathrm{Conv}}
\newcommand{\torus}[0]{(\mathbb{C}^\times)^{N+1}}
\newcommand{\GL}[0]{{\mathrm{GL}(N+1,\mathbb{C})}}

\newcommand{\hurwitz}[0]{{\mathrm{Hu}_{X}}}
\newcommand{\hurpolytope}[0]{{\mathcal{W}(\mathrm{Hu}_{X})}}

\newcommand{\secondarypoly}[0]{{\mathrm{SecPoly}(A)}}

\newcommand{\hyperdiscriminantpolytope}[0]{\mathcal{W}(\Delta_{X\times \mathbb{P}^{n-1}})}
\newcommand{\hyperdiscriminantpolytopeh}[0]{\mathcal{W}_H(\Delta_{X\times \mathbb{P}^{n-1}})}

\newcommand{\polytope}[0]{Q}
\newcommand{\vol}[0]{\mathrm{Vol}_{\mathbb{Z}}}
\newcommand{\hyperdiscriminantpolytopesurface}[0]{\mathcal{W}(\Delta_{X\times \mathbb{P}^{1}})}
\newcommand{\pointupper}[0]{\widetilde{\omega}^{+}}
\newcommand{\pointlower}[0]{\widetilde{\omega}^{-}}
\newcommand{\columnA}[0]{\widetilde{Q}}
\newcommand{\columnAtriangulation}[0]{\widetilde{T}}

\newcommand{\chowpolytope}[0]{\mathcal{W}(R_X)}
\newcommand{\chowpolytopeh}[0]{\mathcal{W}_H(R_X)}

\newcommand{\chowform}[0]{R_X}
\newcommand{\mr}[0]{M_{\mathbb{R}}}

\newtheorem{theorem}{Theorem}[section]
\newtheorem{lemma}[theorem]{Lemma}
\newtheorem{corollary}[theorem]{Corollary}
\newtheorem{proposition}[theorem]{Proposition}

\newtheorem{conjecture}[theorem]{Conjecture}

\theoremstyle{definition}
\newtheorem{example}[theorem]{Example}
\newtheorem{definition}[theorem]{Definition}
\newtheorem{question}[theorem]{Question}
\newtheorem{remark}[theorem]{Remark}

\begin{document}

\title{
	Characteristic vectors for the Hurwitz polytopes of toric varieties
}
\author{Ryoma Ogusu and Yuji Sano}
\address{
Department of Applied Mathematics
    Fukuoka University
    8-19-1 Nanakuma, Jonan-ku, Fukuoka 814-0180, JAPAN
}
\
\email{
sd210002@cis.fukuoka-u.ac.jp
}
\email{
sanoyuji@fukuoka-u.ac.jp
}
\date{}

\thanks{
The authors would like to thank Naoto Yotsutani for sharing his insights to the works of S. Paul.
The second author is supported by JSPS KAKENHI Grant Number 22K03325 and Research funds from Fukuoka University (Grant Number 225001-000).}

%%%abstract%%%
\begin{abstract}
We introduce a characteristic vector with respect to a regular triangulation of the momentum polytope to compute the Hurwitz polytope of a given smooth toric variety.
As a result, we prove that the convex hull of such vectors of all regular triangulations is included in the Hurwitz polytope of a smooth toric surface.
In addition, we discuss the relations of such vectors to $K$-stability of pairs by Paul and toric $K$-stability by Donaldson.
\end{abstract}
%%%%%%%%%%

\maketitle

\section{Introduction} \label{sec:introduction}

Let $\mathbb{G}(k, \mathbb{P}^N)$ denotes the Grassmannian of all $k$-dimensional linear subspaces in the $N$-dimensional projective space $\mathbb{P}^N$.
Let $\variety$ be an $n$-dimensional closed irreducible variety in  $\projectivespace$.
Through this note, we assume that $\variety$ is linearly normal and that both its degree $\degree$ and the codimension of the singular locus of $\variety$ are greater or equal to two.
For $1\le k \le n+1$, generic $(N-k)$-dimensional planes in $\projectivespace$ intersects $\variety$ transversally at any regular point of $\variety$.
The Zariski closure of the set of all $(N-k)$-dimensional planes intersecting $X$ non-transversally at some regular point of $\variety$ constitutes an irreducible variety in $\mathbb{G}(N-k, \mathbb{P}^N)$.
This variety is called the associated variety of $\variety$ in \cite{gkz94}.
In \cite{kohn21}, these varieties are studied as the $(n-k+1)$-th coisotropic variety of $\variety$.
These varieties contain some classical objects: the Chow form (the resultant) if $k=n+1$ and the discriminant if $k=1$.
In \cite{sturmfels17}, Sturmfels shows that the associated variety of $k=n$ is always a hypersurface of $\mathbb{G}(N-n, \mathbb{P}^N)$, and calls its defining polynomial the Hurwitz form.
The overall goal of this note is to find a combinatorial way to compute the weight polytope of the Hurwitz form of a smooth toric variety.

In \cite{sturmfels17}, Sturmfels shows that the Hurwitz form coincides with the discriminant of the Segre embedding of $\variety\times \mathbb{P}^{n-1}$.
The latter is studied by Paul as the hyperdiscriminant of $X$ in the viewpoint of K\"ahler geometry \cite{paul12}. 
Following Tian’s pioneering work \cite{tian94}, Paul (also see  \cite{tian18, paul21}) extends  Geometric Invariant Theory \cite{git94} in terms of the weight polytopes of the Chow form and the hyperdiscriminant of $\variety$.
Hence we will see a relation of the above two polytopes for further study of the stability of Paul in Section \ref{sec:k-stability}.

When $\variety$ is a toric variety $X_A$ associated with a given point configuration $A\subset \mathbb{Z}^n$, we can apply combinatorial tools for the computations of the Chow/discriminant polytopes developed by Gelfand, Kapranov and Zelevinsky \cite{gkz94}. 
In practice, we can carry out such computations by using mathematical software systems, for example, SageMath \cite{sagemath} with TOPCOM \cite{rambau02}
and Macaulay2 \cite{gs}.
Still, however, the computations are a challenging problem in both theoretical and computational viewpoints (cf. \cite{kl20}) if $A$ has many points.

To explain our interest more precisely, let us see an example.
Let $A$ be the seven points in the plane given by
\begin{equation}\label{eq:pointconfiguration}
	A
	=
	\begin{bmatrix}
			0 & 0 & 1 & 1 & 0 & -1 & 0
		\\
			0 & 1 & 1 & 0 & -1 & -1 & -1
		\\
	\end{bmatrix}.
\end{equation}
(Figure \ref{pic:pointconfiguration}).
\begin{figure}
\begin{tikzpicture}
	\fill[black] (0,1)node[above]{(1)} circle(0.05);
	\fill[black] (1,1)node[right]{(2)} circle(0.05);
	\fill[black] (1,0)node[right]{(3)} circle(0.05);
	\fill[black] (0,-1)node[right]{(4)} circle(0.05);
	\fill[black] (-1,-1)node[left]{(5)} circle(0.05);
	\fill[black] (-1,0)node[left]{(6)} circle(0.05);
	\fill[black] (0,0)node[right]{(0)}circle(0.05);
	\draw (0,1)--(1,1)--(1,0)--(0,-1)--(-1,-1)--(-1,0)--cycle;
\end{tikzpicture}	
\caption{Point Configuration $A$}
\label{pic:pointconfiguration}
\end{figure}
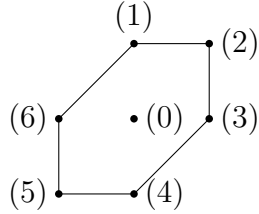
The associated toric variety $\toricvariety$ is the blow up of $\mathbb{P}^2$ along torus invariant three points with anti-cnonical polarization, which is embedded in $\mathbb{P}^6$.
GKZ theory says that the Chow polytope coincides with the secondary polytope $\Sigma(A)$ of $A$ that is the convex hull of the GKZ vector of all regular triangulations of the convex hull $Q$ of $A$ (\cite{ksz92}, \cite{gkz94}).
As for this example, the polygon $(Q,A)$ has $32$ regular triangulations.
On the other hand, the Hurwitz form coincides with the discriminant of the Segre embedding 
$$
	\toricvariety\times \mathbb{P}^1
	\subset
	\mathbb{P}^6\times \mathbb{P}^1
	\hookrightarrow
	\mathbb{P}^{13}.
$$
The variety $\toricvariety\times \mathbb{P}^1$ is toric, and we denote its associated point configuration by $\widetilde{A}$ consisting of the $14$ points in $\mathbb{R}^3$.
By GKZ theory, the discriminant polytope of $\widetilde{A}$ is the convex hull of the massive GKZ vectors of all regular triangulations of $(\columnA, \widetilde{A})$.
As for this example, the polytope $(\columnA, \widetilde{A})$ admits $928930$ regular triangulations. 
This means that the computation of the Hurwitz polytope as the discriminant polytope would be harder than the Chow polytope.
An interest of this note aims at solving the gap between the computations of these two polytopes.

By definition, we can regard the Hurwitz form as an intermediate between the Chow form and the discriminant as the defining polynomial of the associated variety.
Both the Chow polytope and the discriminant polytope are computed from the GKZ vectors and the massive GKZ vectors of regular triangulations of $(Q,A)$ respectively.
This observation raises the following question.
\begin{question}\label{question}
	Is it possible to describe the Hurwitz polytope as the convex hull of some characteristic vectors associated with regular triangulations of $(Q,A)$?	
\end{question}
To solve Question \ref{question}, we introduce a characteristic vector (we call it the Hurwitz vector in Definition \ref{def:hurwitzvector}) as an intermediate vector between the GKZ vector and the massive GKZ vector.
By using these vectors, we give a partial answer\footnote{After completing the first draft version of this note, the second author found another approach to Question \ref{question} by employing some results in K\"ahler geometry. The full answer to Question \ref{question} will be provided in \cite{sano}.} to Question \ref{question}.
\begin{theorem}\label{thm:main_intro}
	Let $A$ be a point configuration $A \subset \mathbb{Z}^{2}$.
	Let $\polytope$ be its convex hull.
	Assume that the associated toric surface $\toricvariety$ is smooth.
	Then the Hurwitz polytope of $\toricvariety$ contains the convex hull of the Hurwitz vectors of all regular triangulations of $(Q,A)$.
\end{theorem}

To compute the Hurwitz polytope exactly, we need to show that the convex hull of all the Hurwitz vectors contains the Hurwitz polytope.
A difficulty to show the converse is that the massive GKZ vectors of all regular triangulations $\columnAtriangulation$ of $(\columnA,\widetilde{A})$ do not necessarily provide the vertices of the Hurwitz polytope.
Hence, the proof should involve the problem: which regular triangulation $\columnAtriangulation$ of ($\columnA,\widetilde{A}$) provides a vertex of the Hurwitz polytope?
This is still a difficult problem even if $\dim\columnA=3$, because ($\columnA,\widetilde{A}$) admits too many triangulations.
Instead of the proof, we collect examples where the converse to Theorem \ref{thm:main_intro} holds in Section \ref{sec:examples}.

Although it is not enough to determine the Hurwitz polytope, Theorem \ref{thm:main_intro} provides the following degree formula of the Hurwitz form in terms of the volume of the polygon $Q$ and its boundary $\partial Q$.
\begin{corollary}\label{cor:degree}
		Let $A$ be a point configuration $A \subset \mathbb{Z}^{2}$.
	Let $\polytope$ be its convex hull.
	Assume that the associated toric surface $\toricvariety$ is smooth.
	Then the degree of the Hurwitz form of $\toricvariety$ in the Pl\"ucker coordinates is equal to
	\begin{equation}\label{eq:degree_hurwitz}
		3\vol(Q)-\vol(\partial Q).
	\end{equation}
	The volume $\vol(Q)$ and $\vol(\partial Q)$ are normalized so that the volume of the fundamental simplex is equal to one.
\end{corollary}
\noindent
Remark that the equality (\ref{eq:degree_hurwitz})  is equivalent to the formula (5.53) in \cite{paul12}  when $\toricvariety$ is a smooth toric surface (see Remark \ref{rem:degree}).

The organization of this note is as follows.
In Section \ref{sec:mainresult}, we recall the terminologies we use in this note.
We refer to the book \cite{gkz94} for most of them.
In particular, we define the Hurwitz vector.
In Section \ref{sec:proof}, we give the proof to Theorem \ref{thm:main_intro}.
As a corollary of Theorem \ref{thm:main_intro}, we compute the degree of the Hurwitz form.
In Section \ref{sec:converse}, we discuss an argument towards the converse to Theorem \ref{thm:main_intro}.
In Section \ref{sec:examples}, we confirm that the converse to Theorem \ref{thm:main_intro} (Conjecture \ref{conj:main}) is true for some examples by computer.
In Section \ref{sec:k-stability}, we discuss relations between the Hurwitz vectors and the stabilities defined in \cite{paul12} and \cite{donaldson02}.

\section{Preliminaries}\label{sec:mainresult}
\subsection{Hurwitz form}

Let $\variety$ be an $n$-dimensional closed irreducible variety in  $\projectivespace$.
Through this note, we assume that $\variety$ is linearly normal and that both its degree $\degree$ and the codimension of the singular locus of $\variety$ are greater or equal to two.
The subvariety
$$
	\{L\in \mathbb{G}(N-n, \mathbb{P}^N)
	\mid\,
	\sharp (L\cap X) < \degree
	\}
$$
has codimension one (Theorem 1.1 \cite{sturmfels17}).
We call its defining polynomial the \textit{Hurwitz form} of $X$.
We denote it by $\hurwitz$.
This is an irreducible element in the coordinate ring of $\mathbb{G}(N-n, \mathbb{P}^N)$.

\subsection{Hyperdiscriminant}

\textit{The dual variety} $\dualvariety$ is the subvariety of $(\projectivespace)^\vee$ consisting of the hyperplanes tangent to $\variety$, i.e., the Zariski closure of 
$$
	\{H\in(\projectivespace)^\vee \mid\, \mathbb{T}_p \variety\subset H \,\, \mbox{for some regular point } p\in X\}.
$$
Here $\mathbb{T}_p \variety \subset\projectivespace$ denotes the embedded tangent space of $\variety\subset \projectivespace$.
If $\dualvariety$ has the codimension one, then we call its defining polynomial \textit{the discriminant} of $\variety$.
In general, however, the codimension of $\dualvariety$ may not be equal to one. 

If $\variety$ is non-degenerate and the codimension of the singular part of $\variety$ is greater or equal to two, then Theorem 2.3 in \cite{zak93} and an application of Katz dimension formula (cf. Corollary 5.9 in Chapter 1 \cite{gkz94}) imply that the dual variety of the Segre embedding of $\segreembedding$ into $\mathbb{P}^{(N+1)n-1}$ is always a hypersurface in $(\mathbb{P}^{(N+1)n-1})^\vee$.
Its defining polynomial is called \textit{the hyperdiscriminant} of $\variety$ in \cite{paul12}.
We denote it by $\hyperdiscriminant$.

Writing $\hurwitz$ in the Stiefel coordinates, Sturmfels \cite{sturmfels17} proves that $\hurwitz$ is equal to $\hyperdiscriminant$. 
We also refer to Section 3.2E, Chapter 3 in \cite{gkz94} and Proposition 5 in \cite{kohn21} for the proof.
This coincidence can be seen as a variant of the Cayley trick (cf. Corollary 2.8 in Chapter 2 \cite{gkz94}).
Regarding $\hurwitz$ as $\hyperdiscriminant$, we are able to apply GKZ theory to it directly.
This will be discussed in Section \ref{sec:proof}.

\subsection{Hurwitz polytope}
We refer to \cite{sturmfels17} for this subsection.
Let us consider the natural action on $\projectivespace$ of the $(N+1)$-dimensional torus $\torus$ represented by
$$
	\begin{pmatrix}
		t_{1} &  &  \\
		 & \ddots & \\
		 && t_{N+1}
	\end{pmatrix}
	\in \torus \subset \GL.
$$
This is extended to the action on $\mathbb{G}(N-n,N)$ in a natural way.
This action induces the $\mathbb{Z}^{N+1}$-grading of the Pl\"ucker coordinate  ring of $\mathbb{G}(N-n,N)$ by
\begin{equation}\label{eq:degree_plucker}
	\mathrm{deg}(p_{i_1\cdots i_n})
	= \sum_{j=1}^n \mathbf{e}^{(N+1)}_{i_j}
	\in \mathbb{Z}^{N+1},
\end{equation}
where $1\le i_{1},\ldots, i_{n}\le N+1$ and $\mathbf{e}^{(N+1)}_i$ ($1\le i \le N+1$) denote the standard basis of $\mathbb{R}^{N+1}$.
We call the weight polytope of $\hurwitz$ with respect to this action \textit{the Hurwitz polytope} denoted by $\hurpolytope$.
This is the object we desire to compute.

On the other hand, the weight polytope of $\hyperdiscriminant$ is defined in the following way. 
The torus action on $\projectivespace$ is extended to the action on $\projectivespace\times \mathbb{P}^{n-1}$ by acting on the first factor in a natural way and on the second factor trivially.
Then, the weight decomposition of a polynomial on $(\mathbb{P}^{(N+1)n-1})^\vee$ with respect to the action on $\projectivespace\times \mathbb{P}^{n-1}$  is equivalent to the $\mathbb{Z}^{N+1}$-grading of the Pl\"ucker coordinate  ring of $\mathbb{G}(N-n,N)$ defined by (\ref{eq:degree_plucker}).
Hence, the weight polytope $\hyperdiscriminantpolytope$ coincides with $\hurpolytope$.

Remark that the torus that appeared here is different from the one considered in \cite{paul12}. 
The weight polytope in \cite{paul12} is defined with respect to the action of $N$-dimensional torus in $\mathrm{SL}(N+1,\mathbb{C})$ (see the proof of Corollary \ref{cor:ksemistability}).

\subsection{Triangulations of Point Configurations}
Let 
$$
	A=\{\omega_1,\ldots, \omega_{N+1}\}\subset \mathbb{Z}^n
$$
be a lattice point configuration on $\mathbb{R}^{n}$.
Assume that the dimension of the convex hull $\polytope$ of $A$ is equal to $n$.
For a point configuration $A$, we define an $n$-dimensional projective toric variety $\toricvariety\subset\projectivespace$ by the Zariski closure of 
$$
	\bigg\{
	(t^{\omega_1}:\ldots:t^{\omega_{N+1}})\in
	\projectivespace
	\bigg|\,\,
	t=(t_1,\ldots,t_{n})\in (\mathbb{C}^{\times})^n
	\bigg\} 
$$
where $\omega_i=(\omega_{i1},\ldots, \omega_{in})\in \mathbb{Z}^{n}$ and $t^{\omega_i}=\prod_{j=1}^n t_j^{\omega_{ij}} $.

Let $T$ be \textit{a triangulation} of $(Q,A)$, i.e., a collection $\Sigma_T$ of simplices whose vertices in $A$ such that the support $|\Sigma_T|=Q$ and any intersection of two simplices in $\Sigma_T$ is contained in $\Sigma_T$, and any face of the intersection is also in $\Sigma_T$.
We denote the set of $k$-dimensional simplices in $\Sigma_T$ by $\Sigma_T(k)$.
We say that $T$ (as a subdivision) is \textit{regular} if it is obtained by the vertical projection of the lower convex hull of the lifted point configuration 
$$
	A^g:= \{(\omega_1,g(\omega_1)),\ldots,(\omega_{N+1},g(\omega_{N+1}))\}\subset \mathbb{R}^{n+1}
$$
by some concave piecewise-linear function $g:Q\to\mathbb{R}$, i.e., $g$ is linear on every simplex in $\Sigma_T(n)$.

\subsection{GKZ vector}
For $\sigma\in\Sigma_T(k)$, we say that $\sigma$ is \textit{massive} if $\sigma$ is contained in some $k$-dimensional face of $Q$.
Remark that we define that any maximal simplex is massive.
For a given triangulation $T$, the following vector is called \textit{the massive GKZ vector} ((3.2) in Chapter 11 \cite{gkz94}):
\begin{equation}\label{eq:massiveGKZ}
	\eta_T:= \sum_{k=0}^{n}(-1)^{n-k}\eta_{T,k}\in \mathbb{Z}^{N+1} 	
\end{equation}
where
$$
\eta_{T,k}= (\eta_{T,k}(\omega_1),\ldots,\eta_{T,k}(\omega_{N+1}))
$$
and
\begin{equation}\label{eq:gkzvector}
	\eta_{T,k}(\omega_i)= \sum_{\omega_i\prec\sigma}\vol(\sigma).	
\end{equation}
In (\ref{eq:gkzvector}), a simplex $\sigma$ runs through all the massive simplices in $\Sigma_T(k)$ containing $\omega_i$ as a vertex.
Remark that the volume $\vol(\sigma)$ is normalized so that the volume of the fundamental simplex is equal to one.

In particular, the vector $\eta_{T,n}\in\mathbb{Z}^{N+1}$ is called \textit{the GKZ vector} of $T$ ((1.4) in Chapter 7 in \cite{gkz94}), whose $i$-th coordinate is obtained as the sum of the volume of all maximal simplices containing the lattice point $\omega_i$.

\subsection{Secondary Polytopes}
We call the convex hull of the GKZ vectors $\eta_{T,n}$ for all triangulations $T$ of $(Q,A)$ \textit{the secondary polytope} of $A$.
We denote it by $\secondarypoly$.
By \cite{ksz92} (also see Theorem 3.1 in Chapter 8 \cite{gkz94}), the polytope $\secondarypoly$ coincides with the Chow polytope of $\toricvariety$, which is the weight polytope of the Chow form of $\toricvariety$ with respect to the action of $\torus$.
The vertices of the secondary polytope $\secondarypoly$ are in one-to-one correspondence with the regular triangulations of $(Q,A)$.
In particular, $\eta_{T,n}= \eta_{T’,n}$ if and only if $T=T’$.
We refer to Chapter 7 in \cite{gkz94} for the properties of the secondary polytopes.

\subsection{Discriminant Polytopes}
The massive GKZ vectors are \textit{not} in one-to-one correspondence with the regular triangulations of $(Q,A)$.
We say that a triangulation $T$ is \textit{$D$-equivalent} to another triangulation $T’$ if $\eta_{T}=\eta_{T’}$.
Then the following holds.
\begin{theorem}[Theorem 3.2 in Chapter 11 \cite{gkz94}]\label{thm:discriminant}
The vertices of the Newton polytope of the discriminant (if it exists) of a smooth toric variety $\toricvariety$ correspond exactly to the massive GKZ vectors of the $D$-equivalent classes of the regular triangulations of $(Q,A)$
\end{theorem}
Remark that the Newton polytope of the $\toricvariety$-discriminant is equal to its weight polytope with respect to the natural action of $\torus$ on $(\projectivespace)^\vee$.

\subsection{Hurwitz vector}
Regarding $\hurwitz$ as the hyperdiscriminant, we can compute its weight polytope by GKZ theory.
By Theorem \ref{thm:discriminant}, $\hyperdiscriminantpolytope$ can be calculated from the massive GKZ vectors of all  regular triangulations of not $A$ but 
$$
	\widetilde{A}=
	\{
	(\omega_i, \mathbf{e}^{(n-1)}_j)\in\mathbb{R}^{n}\times\mathbb{R}^{n-1}
	\mid\,
	\omega_i\in A, 0\le j \le n-1
	\}.
$$
The set $\widetilde{A}$ is the cartesian product of $A$ and the vertices of the $(n-1)$-dimensional unit simplex.
In above, $\mathbf{e}^{(n-1)}_j$ $(1\le i \le n-1)$ denote the standard unit vectors of $\mathbb{R}^{n-1}$ and $\mathbf{e}^{(n-1)}_0= \mathbf{o}$.
A difficulty in this computation is that the number of all regular triangulations of $(\columnA, \widetilde{A})$ is quite larger than the one of $(Q,A)$.

As mentioned in Introduction, we propose the following question.
\begin{question}\label{que:hurwitz}
	Can we obtain the Hurwitz polytope $\hurpolytope$ from the regular triangulations of $(Q,A)$ instead of $(\columnA, \widetilde{A})$ directly?
\end{question}
\noindent
To consider the above question, let us recall the followings:
\begin{itemize}
	\item 
		The Chow polytope of $\toricvariety$ coincides with the secondary polytope $\secondarypoly$ whose vertices are given by the GKZ vectors $\eta_{T,n}$.
	\item
		The vertices of the discriminant polytope of $\toricvariety$ are given by the massive GKZ vectors defined by (\ref{eq:massiveGKZ}).
	\item
		The Hurwitz form is an intermediate between the Chow form and the discriminant.
\end{itemize}
From the above, we introduce an intermediate vector between $\eta_{T,n}$ and $\eta_T$ as follows.
\begin{definition}\label{def:hurwitzvector}
	Let $A$ be a point configuration on $\mathbb{Z}^n$.
	For a (not necessarily regular) triangulation $T$ of $(Q,A)$, we define
	$$
		\xi_{T}:= n\eta_{T,n} -\eta_{T,n-1}.
	$$
	We call it \textit{the Hurwitz vector} with respect to $T$.
\end{definition}

\section{Proof}\label{sec:proof}

\subsection{Proof of Theorem \ref{thm:main_intro}}
Through this section, we let $n=2$.
In addition, we assume that $\toricvariety$ is smooth to apply Theorem \ref{thm:discriminant}.
We prove Theorem \ref{thm:main_intro} in the following two steps:
\begin{enumerate}
	\item 
		construct some regular triangulation $\columnAtriangulation$ of $(\columnA,\widetilde{A})$ for a given regular triangulation $T$ of $(Q,A)$, then
	\item
		verify that $\xi_T$ is equal to the associated point $\nu_{\widetilde{T}}$ of $\hyperdiscriminantpolytopesurface$ with respect to the above triangulation $\columnAtriangulation$.				
\end{enumerate}

Let us see how the vertices of $\hyperdiscriminantpolytopesurface$ are obtained from the massive GKZ vectors of a regular triangulation $\columnAtriangulation$ of $(\columnA,\widetilde{A})$.
Let 
$$
	A=
	\{
		\omega_1,\ldots,\omega_{N+1}
	\}
	\subset \mathbb{Z}^2.
$$ 
For each $\omega_i\in A$, we associate the following two points
$$
	\pointupper_i:=(\omega_i,0),\quad
	\pointlower_i:=(\omega_i,1)
$$
in $\mathbb{Z}^3$.
Let
$$
	\widetilde{A}
	=\{\pointlower_i\mid\, 1\le i\le N+1\}
	\cup 
	\{\pointupper_i\mid\, 1\le i\le N+1\}
	\subset \mathbb{Z}^3.
$$
Let $\columnA$ be the prism $Q\times I$ in $\mathbb{R}^3$, which is equal to the convex hull of $\widetilde{A}$.
For a triangulation $\columnAtriangulation$ of $(\columnA, \widetilde{A})$, we define the vector $\nu_{\columnAtriangulation}$ by
\begin{equation*}\label{eq:vertices_from_mGKZ}
	\nu_{\columnAtriangulation}
	=
	(\nu_{\columnAtriangulation, 1},
	\ldots,
	\nu_{\columnAtriangulation, N+1}
	)
	\in\mathbb{Z}^{N+1}
\end{equation*}
where
\begin{eqnarray*}
		\nu_{\columnAtriangulation, i}
	&=&
		\eta_{\columnAtriangulation}(\pointupper_i)
		+
		\eta_{\columnAtriangulation}(\pointlower_i)
	\\
	&=& 
		\sum_{j=0}^{3}
		(-1)^j
		\big(
		\eta_{\columnAtriangulation,j}(\pointupper_i)
		+
		\eta_{\columnAtriangulation,j}(\pointlower_i)
		\big).	
\end{eqnarray*}
\begin{proposition}\label{prop:vertices_nuvector}
Any vertex of $\hyperdiscriminantpolytopesurface$ is equal to the vector $\nu_{\columnAtriangulation}$ for some regular triangulation $\columnAtriangulation$ of $(\columnA, \widetilde{A})$.
In particular, $\hyperdiscriminantpolytopesurface$ is the convex hull of $\nu_{\columnAtriangulation}$ of all regular triangulations $\columnAtriangulation$.
\end{proposition}
\begin{proof}
Recall that $\sum_{j}(-1)^{j}\eta_{\columnAtriangulation,j}(\pointupper_i)$ and $\sum_{j}(-1)^{j}\eta_{\columnAtriangulation,j}(\pointlower_i)$ are the vertices of the Newton polytope $\mathcal{D}(\hyperdiscriminantsurface)$ of $\hyperdiscriminantsurface$ when we regard $\hyperdiscriminantsurface$ as a polynomial on $(\mathbb{P}^{2(N+1)-1})^{\vee}$.
They are also equal to the vertices of the weight polytope of $\hyperdiscriminantsurface$ with respect to the natural action of $(\mathbb{C}^{\times})^{2(N+1)}$.
We label the elements in $\widetilde{A}$ by
$$
	(\pointupper_{1},\,
	\ldots,\,
	\pointupper_{N+1},\,
	\pointlower_{1},\,
	\ldots,
	\pointlower_{N+1}).
$$
Let the projection $\pi:\mathbb{R}^{2(N+1)}\to \mathbb{R}^{N+1}$ be
$$
	(
	x^{+}_{1},\, \ldots, x^{+}_{N+1},\,
	x^{-}_{1},\, \ldots, x^{-}_{N+1}
	)
	\mapsto
	(
	x^{+}_{1}+x^{-}_{1},\, \ldots, x^{+}_{N+1}+x^{-}_{N+1}
	).
$$
Since the torus $\torus$ act trivially on the second factor of $\projectivespace\times\mathbb{P}^{1}$, the weight can be written by
$$
	t_{1}^{\eta_{\columnAtriangulation}(\pointupper_1)}
	\cdots
	t_{N+1}^{\eta_{\columnAtriangulation}(\pointupper_{N+1})}
	t_{1}^{\eta_{\columnAtriangulation}(\pointlower_1)}
	\cdots
	t_{N+1}^{\eta_{\columnAtriangulation}(\pointlower_{N+1})}
	=
	t^{\nu_{\columnAtriangulation}}
$$
where $t=(t_{1},\ldots, t_{N+1})\in (\mathbb{C}^{\times})^{N+1}$.
This implies that 
$$
	\pi (\mathcal{D}(\hyperdiscriminantsurface))
	=\hyperdiscriminantpolytopesurface.
$$
Since the inverse image under $\pi$ is an affine subspace of dimension $(N+1)$ in $\mathbb{R}^{2(N+1)}$, the inverse image $\pi^{-1}(v)$ for any vertex $v$ of $\hyperdiscriminantpolytopesurface$ contains some vertex of $\mathcal{D}(\hyperdiscriminantsurface)$.
This means that there exists some regular triangulation $\widetilde{T}$ of $(\columnA, \widetilde{A})$ such that $v= \nu_{\widetilde{T}}$.
The proof is completed.
\end{proof}
Notice that a regular triangulation $\widetilde{T}$ of $(\columnA, \widetilde{A})$ whose $\nu_{\widetilde{T}}$ is a vertex of $\hyperdiscriminantpolytopesurface$ may not be unique.
Following the $D$-equivalence, we introduce the following equivalence among the regular triangulations of $(\columnA, \widetilde{A})$.
\begin{definition}
	A regular triangulation $\columnAtriangulation$ of $(\columnA,\widetilde{A})$ is $H$-equivalent to another regular triangulation $\columnAtriangulation’$ if $\nu_{\columnAtriangulation}=\nu_{\columnAtriangulation’}$.
\end{definition}
\noindent
The $D$-equivalence implies the $H$-equivalence.

Now, we construct a regular triangulation of $(\columnA,\widetilde{A})$ for each regular triangulation $T$ of $(Q,A)$.
First, we subdivide $\columnA$ into $\triangle\times I$ with respect to each simplex $\triangle \in\Sigma_T(2)$, i.e.,
\begin{equation*}
	\columnA
	= 
	\bigcup_{\triangle\in\Sigma_T(2)}
	\triangle\times I.
\end{equation*}
By construction, this subdivision of $\columnA$ is regular, because $T$ is regular.
Next, the subdivision can be refined to a regular triangulation $\columnAtriangulation$ of $(\columnA,\widetilde{A})$ (see \cite{drs10}).
We notice that $\columnAtriangulation$ subdivides both of the upper facet $Q\times \{1\}$ and the lower facet $Q\times \{0\}$ by the same triangulation $T$.

\begin{definition}
	For a regular triangulation $A$ of $(Q,A)$, we call the regular triangulation $\columnAtriangulation$ of $(\columnA, \widetilde{A})$ as above \textit{the vertical regular triangulation} associated with $T$.
\end{definition}
\noindent
Remark that the vertical regular triangulation $\columnAtriangulation$ may not be unique with respect to a given regular triangulation $T$.

\begin{lemma}\label{lem:simplex_hurwitz}
	Suppose that $A\subset\mathbb{Z}^{2}$ consists of three points $\{\omega_1,\,\omega_2,\omega_3\}$ constituting a two dimensional simplex $Q=\triangle$.
	Let $\columnAtriangulation$ be a regular triangulation of the triangular prism $(\columnA, \widetilde{A})$.
	Then, we have
	\begin{equation}\label{eq:simplex_volume_prism}
			\eta_{\columnAtriangulation,3}(\pointupper_i)+\eta_{\columnAtriangulation,3}(\pointlower_i)
		=
			4 \vol(\triangle).		
	\end{equation}
	In particular, the left hand is independent of the choice of the triangulation $\columnAtriangulation$.
\end{lemma}
\begin{proof}
Notice that $\columnA$ has six regular triangulations.
The statement follows by calculating the left hand in (\ref{eq:simplex_volume_prism}) for each triangulation directly.
\end{proof}

The following proposition completes the proof of Theorem \ref{thm:main_intro}.
\begin{proposition}\label{prop:enumerate}
	Let $\columnAtriangulation$ be any vertical regular triangulation associated with a given regular triangulation $T$ of $(Q,A)$.
	Then the Hurwitz vector for $T$
	$$
		\xi_{T}=
		2\eta_{T,2}-\eta_{T,1}
	$$
	is equal to the lattice point $\nu_{\columnAtriangulation}$ of $\hyperdiscriminantpolytopesurface$ corresponding to $\columnAtriangulation$.
\end{proposition}
\begin{proof}
We show that 
$$
	\xi_{T}(\omega_i):= 2\eta_{T,2}(\omega_i)-\eta_{T,1}(\omega_i)
$$
is equal to $\nu_{\columnAtriangulation,i}$ for each $\omega_i$ in the following three cases separately:  
\begin{itemize}
	\item the case where $\omega_i$ is a vertex of $Q$, 
	\item the case where $\omega_i$ is an interior point of some edge of $Q$, and
	\item the case where $\omega_i$ is an interior point of $Q$.
\end{itemize}

For the first case, the statement follows from
\begin{eqnarray*}
		\eta_{\columnAtriangulation,3}(\pointlower_i)
		+
		\eta_{\columnAtriangulation,3}(\pointupper_i)
	&=&	
		4\eta_{T,2}(\omega_i),
	\\
		\eta_{\columnAtriangulation,2}(\pointlower_i)
		+
		\eta_{\columnAtriangulation,2}(\pointupper_i)
	&=&		
		2\eta_{T,2}(\omega_i)+3\eta_{T,1}(\omega_i),
	\\
		\eta_{\columnAtriangulation,1}(\pointlower_i)
		+
		\eta_{\columnAtriangulation,1}(\pointupper_i)
	&=&	
		2\eta_{T,1}(\omega_i)+2,	\\
		\eta_{\columnAtriangulation,0}(\pointlower_i)
		+
		\eta_{\columnAtriangulation,0}(\pointupper_i)
	&=&	
		2.
\end{eqnarray*}
The first equality in the above follows from Lemma \ref{lem:simplex_hurwitz}.

For the second case, the statement follows from
\begin{eqnarray*}
		\eta_{\columnAtriangulation,3}(\pointlower_i)
		+
		\eta_{\columnAtriangulation,3}(\pointupper_i)
	&=&	
		4\eta_{T,2}(\omega_i),
	\\
		\eta_{\columnAtriangulation,2}(\pointlower_i)
		+
		\eta_{\columnAtriangulation,2}(\pointupper_i)
	&=&		
		2\eta_{T,2}(\omega_i)+3\eta_{T,1}(\omega_i),
	\\
		\eta_{\columnAtriangulation,1}(\pointlower_i)
		+
		\eta_{\columnAtriangulation,1}(\pointupper_i)
	&=&	
		2\eta_{T,2}(\omega_i),
	\\
		\eta_{\columnAtriangulation,0}(\pointlower_i)
		+
		\eta_{\columnAtriangulation,0}(\pointupper_i)
	&=&	
		0.
\end{eqnarray*} 

For the third case, the statement follows from
\begin{eqnarray*}
		\eta_{\columnAtriangulation,3}(\pointlower_i)
		+
		\eta_{\columnAtriangulation,3}(\pointupper_i)
	&=&	
		4\eta_{T,2}(\omega_i),
	\\
		\eta_{\columnAtriangulation,2}(\pointlower_i)
		+
		\eta_{\columnAtriangulation,2}(\pointupper_i)
	&=&		
		2\eta_{T,2}(\omega_i),
	\\
		\eta_{\columnAtriangulation,1}(\pointlower_i)
		+
		\eta_{\columnAtriangulation,1}(\pointupper_i)
	&=&	
		0,	
	\\
		\eta_{\columnAtriangulation,0}(\pointlower_i)
		+
		\eta_{\columnAtriangulation,0}(\pointupper_i)
	&=&	
		0,
\end{eqnarray*} 
and $\eta_{T,1}(\omega_i)=0$.
Therefore, the proof is completed. 
\end{proof}
\noindent
Remark that Proposition \ref{prop:enumerate} still holds if $\columnAtriangulation$ is not regular.

\subsection{Proof of Corollary \ref{cor:degree}}
Theorem \ref{thm:main_intro} (with the equality (\ref{eq:degree_plucker})) implies that the degree of the Hurwitz form of a toric surface $\toricvariety$ in the Pl\"ucker coordinates is equal to 
\begin{equation}\label{eq:degree-1}
	\frac{1}{2}\sum_{i=1}^{N+1} \xi_T(\omega_i)
	=
	\sum_{i=1}^{N+1} \eta_{T,2}(\omega_i)
	- \frac{1}{2}\sum_{i=1}^{N+1} \eta_{T,1}(\omega_i)	
\end{equation}
for some (any) triangulation $T$.

Any simplex in $\Sigma_{T}(2)$ appears three times in the first summation in the right hand of (\ref{eq:degree-1}) because the simplex has three vertices.
Hence, we have
\begin{equation}\label{eq:degree-2}
	\sum_{i=1}^{N+1} \eta_{T,2}(\omega_i)
	= 3 \vol(Q).
\end{equation}

Since any massive edge in $\Sigma_{T}(1)$ are shared among its two endpoints, we have
\begin{equation}\label{eq:degree-3}
	\sum_{i=1}^{N+1} \eta_{T,1}(\omega_i)
	= 2 \vol(\partial Q).
\end{equation}
The statement follows from (\ref{eq:degree-1}), (\ref{eq:degree-2}) and (\ref{eq:degree-3}).
The proof is completed.

\begin{remark}\label{rem:degree}
The equality (\ref{eq:degree_hurwitz}) in Corollary \ref{cor:degree} is equivalent to the formula (5.53) in \cite{paul12}
$$
	\mathrm{deg}(\Delta_{\toricvariety\times \mathbb{P}^1})
	= d_{\toricvariety}(6-\mu) 
$$
in the case where $\toricvariety$ is a smooth toric surface, where $d_{\toricvariety} $ is the degree of $\toricvariety$ and $\mu$ is the average of the scalar curvature of $\toricvariety$.
Let us see it.
Since $\toricvariety$ is smooth, the degree $d_{\toricvariety}$ is equal to $\vol(Q)$.
Comparing the Riemann-Roch Theorem with the expansion of the Ehrhart polynomial, we find that the average $\mu$ of the scalar curvature is equal to $2\vol(\partial Q)/\vol(Q)$.
Applying the formula (5.53) in \cite{paul12}, we get (\ref{eq:degree_hurwitz}) modulo the factor $n=2$, because the formula in \cite{paul12} is written in the Steifels coordinates, whereas (\ref{eq:degree_hurwitz}) is written in the Pl\"ucker coordinates.
Also see the equality (\ref{eq:degree_hurwitz_general}) for the case of general dimension.
\end{remark}

\section{Converse to Theorem \ref{thm:main_intro}}\label{sec:converse}

To give a complete answer to Question \ref{que:hurwitz}, we need to prove the converse to Theorem \ref{thm:main_intro}, i.e., the convex hull of the Hurwitz vectors of all regular triangulations of $(Q,A)$ contains the Hurwitz polytope.
It suffices to prove that if $\columnAtriangulation$ is not $H$-equivalent to any vertical triangulation, then the vector $\nu_{\columnAtriangulation}$ is not equal to any vertex of $\hyperdiscriminantpolytopesurface$.
Although we cannot prove it in this note, we will give a typical example of $\columnAtriangulation$ whose $\nu_{\columnAtriangulation}$ is not equal to any vertex of $\hyperdiscriminantpolytopesurface$ instead of the proof of the converse to Theorem \ref{thm:main_intro}.

Let us recall a modification of a triangulation.
Let $Z\subset\Sigma_{\columnAtriangulation}(0)$ be \textit{a circuit}, i.e., any proper subset of $Z$ constitutes a simplex of some dimension, but $Z$ itself is not linearly independent.
Then the convex hull $\conv(Z)$ has only two triangulations $\columnAtriangulation ^+_Z$ and $\columnAtriangulation ^-_Z$.
If $\columnAtriangulation ^+_Z$ gives the given triangulation $\columnAtriangulation$, then $\columnAtriangulation ^-_Z$ gives another triangulation $T’$.
We call $\columnAtriangulation’$ \textit{the modification of $\columnAtriangulation $ along $Z$} (Chapter 7 \cite{gkz94}).
We denote it by $s_Z(\columnAtriangulation)$.

\begin{definition}\label{def:mixedsimplex}
	We say that a $3$-simplex $\triangle$ in $\Sigma_{\columnAtriangulation}(3)$ is \textit{mixed} when two of the vertices of $\triangle$ lie on the upper facet $Q\times\{1\}$ and the other two lie on the lower facet $Q\times\{0\}$.
\end{definition}
\noindent
For a mixed simplex $\triangle$, we denote the two vertices on the upper facet by $\pointupper_i,\,\pointupper_j$ and denote the other two vertices by $\pointlower_{i’},\pointlower_{j’}$.

\begin{definition}\label{def:mixedsimplexcubic}
	We say that a mixed simplex $\triangle$ is \textit{cubic} if and only if the following four simplices
	$$
		(\pointupper_{i},\pointlower_{i’},\pointlower_{j’},\pointlower_{i}),
		\,\,
		(\pointupper_{j},\pointlower_{i’},\pointlower_{j’},\pointlower_{j}),
	$$
	$$
		(\pointupper_{i},\pointupper_{j},\pointupper_{i’},\pointlower_{i’}),
		\,\,
		(\pointupper_{i},\pointupper_{j},\pointupper_{j’},\pointlower_{j’})
	$$
	are contained in $\Sigma _{\columnAtriangulation}(3)$ (See Figure \ref{pic:cubicmixedsimplex}).
\end{definition}

\begin{figure}
\begin{tikzpicture}
	\draw (0,0)node[left]{$\pointlower_{i}$}--(5,0) node[below]{$\pointlower_{i’}$}--(7,1)node[right]{$\pointlower_{j}$};
	\draw[dashed] (7,1)--(2,1)--(0,0);
	\draw (1.75,1)node[left]{$\pointlower_{j’}$};
	\draw[shift={(0,4)}] (0,0)node[left]{$\pointupper_{i}$}--(5,0) node[above]{$\pointupper_{i’}$}--(7,1)node[right]{$\pointupper_{j}$};
	\draw[shift={(0,4)}] (7,1)--(2,1)node[above]{$\pointupper_{j’}$}--(0,0);
	\draw (0,0)--(0,4);
	\draw[shift={(5,0)}] (0,0)--(0,4);
	\draw[shift={(7,1)}] (0,0)--(0,4);
	\draw[shift={(2,1)},dashed] (0,0)--(0,4);
	\draw (0,4)--(7,5);
	\draw[dashed] (5,0)--(2,1);
	\draw (5,0)--(7,5);
	\draw[dashed] (2,1)--(0,4);
	\draw[dashed] (2,1)--(7,5);
	\draw (5,0)--(0,4);	
\end{tikzpicture}	
\caption{Cubic mixed simplex $(\pointupper_i,\pointupper_j,\pointlower_{i’},\pointlower_{j’})$}
\label{pic:cubicmixedsimplex}
\end{figure}

\begin{proposition}\label{prop:nonvertex}
	If a regular triangulation $\columnAtriangulation$ of $(\columnA,\widetilde{A})$ admits a cubic mixed simplex $\triangle$, then the associated vector $\nu_{\columnAtriangulation}$ is not any vertex of $\hyperdiscriminantpolytopesurface$.
\end{proposition}
\begin{proof}
Let 
\begin{eqnarray*}
			\vol((\pointupper_i,\pointupper_j,\pointupper_{i’}))
			=	
			a,
		&&		
			\vol((\pointupper_i,\pointupper_j,\pointupper_{j’}))
			=
			b,
		\\
			\vol((\pointupper_{i},\pointlower_{i’},\pointlower_{j’}))
			=	
			c,
		&&		
			\vol((\pointupper_{j},\pointlower_{i’},\pointlower_{j’}))
			=
			d.		
\end{eqnarray*}
By definition, we have
$$
	\vol(\triangle)
	=\vol((\pointupper_{i},\pointupper_{j},\pointlower_{i’},\pointlower_{j’}))
	=a+b=c+d.
$$
Let $Z_1$ be the circuit consisting of 
$$
	\{\pointupper_i,\pointupper_j,\pointupper_{i’},\pointlower_{i’},\pointlower_{j’}\}.
$$
Then we have
\begin{eqnarray*}
		\nu_{s_{Z_1}(\columnAtriangulation),i}
		=
		\nu_{\columnAtriangulation,i}-d,
	&&
		\nu_{s_{Z_1}(\columnAtriangulation),j}
		=
		\nu_{\columnAtriangulation,j}-c,
	\\
		\nu_{s_{Z_1}(\columnAtriangulation),i’}
		=
		\nu_{\columnAtriangulation,i’}+b,
	&&
		\nu_{s_{Z_1}(\columnAtriangulation),i}
		=
		\nu_{\columnAtriangulation,i}+a.
\end{eqnarray*}
For the other vectors, we have
$$
	\nu_{s_{Z_1}(\columnAtriangulation),k}
	=
	\nu_{\columnAtriangulation,k},
	\,\, k\neq i, j, i’, j’.
$$
Similarly, for the circuit $Z_2$ defined by
$$
	\{\pointupper_i,\pointupper_j,\pointlower_{i},\pointlower_{i’},\pointlower_{j’}\},
$$
we have
\begin{eqnarray*}
		\nu_{s_{Z_2}(\columnAtriangulation),i}
		=
		\nu_{\columnAtriangulation,i}+d,
	&&
		\nu_{s_{Z_2}(\columnAtriangulation),j}
		=
		\nu_{\columnAtriangulation,j}+c,
	\\
		\nu_{s_{Z_2}(\columnAtriangulation),i’}
		=
		\nu_{\columnAtriangulation,i’}-b,
	&&
		\nu_{s_{Z_2}(\columnAtriangulation),i}
		=
		\nu_{\columnAtriangulation,i}-a,
\end{eqnarray*}
and the other vectors are not changed.
Therefore, $\nu_{\columnAtriangulation}$ is the middle point of the segment from $\nu_{s_{Z_1}(\columnAtriangulation)}$ to $\nu_{s_{Z_2}(\columnAtriangulation)}$.
The proof is completed.
\end{proof}	

A regular triangulation $\columnAtriangulation$ with a cubic mixed simplex is a typical example of the non-vertical cases.
However, there exist non-vertical triangulations $\columnAtriangulation$, whose vector $\nu_{\columnAtriangulation}$ is not equal to the middle point of any edge of $\hyperdiscriminantpolytopesurface$.
\begin{example}
Let $A$ be the point configuration defined by (\ref{eq:pointconfiguration}).
Following the labelling shown in Figure \ref{pic:pointconfiguration}, we label the point of $\widetilde{A}$ so that if $(i_+)$ (resp. $(i_-)$) indicates the lattice point on the upper (resp. lower) facet of $\columnA$ whose projection to $A$ is equal to the point $(i)$.
Let $\columnAtriangulation$ be the regular triangulation consisting of the following simplices:
\begin{eqnarray*}
	&&
		\langle 1_+,2_+,6_+,2_- \rangle,\, 
		\langle 2_+,3_+,4_+,3_- \rangle,\, 
		\langle 2_+,4_+,5_+,2_- \rangle,\, 
		\langle 2_+,4_+,2_-,3_- \rangle,\, 
	\\
	&&
		\langle 2_+,5_+,6_+,2_- \rangle,\,
		\langle 3_+,4_+,3_-,4_- \rangle,\,
		\langle 4_+,5_+,1_-,2_- \rangle,\, 
		\langle 4_+,5_+,1_-,6_- \rangle,\,
	\\
	&&
		\langle 4_+,1_-,2_-,4_- \rangle,\,
		\langle 4_+,1_-,4_-,5_- \rangle,\, 
		\langle 4_+,1_-,5_-,7_- \rangle,\,
		\langle 4_+,2_-,3_-,4_- \rangle,\,
	\\
	&&
		\langle 5_+,6_+,2_-,6_- \rangle,\,
		\langle 5_+,1_-,2_-,6_- \rangle,\,
		\langle 6_+,2_-,6_-,7_- \rangle.
\end{eqnarray*}
Then the associated vector $\nu_{\columnAtriangulation}$ is equal to 
$$
	(2, 7, 4, 0, 7, 4, 0).
$$
Although this vector is neither any vertex of $\hyperdiscriminantpolytopesurface$ nor any middle point of them as Proposition \ref{prop:nonvertex}, the vector is contained in the convex of the Hurwitz vectors.
The list of the Hurwitz vectors of this example is as follows:
\begin{eqnarray*}
	&&
		(12, 2, 2, 2, 2, 2, 2),
		(10, 0, 4, 2, 2, 2, 4),
		(10, 2, 2, 2, 4, 0, 4),
		(10, 2, 2, 4, 0, 4, 2),
	\\
	&&
		(10, 2, 4, 0, 4, 2, 2),
		(10, 4, 0, 4, 2, 2, 2),
		(10, 4, 2, 2, 2, 4, 0),
		(8, 0, 4, 2, 4, 0, 6),
	\\
	&&
		(8, 0, 4, 4, 0, 4, 4),
		(8, 0, 6, 0, 4, 2, 4),
		(8, 2, 4, 0, 6, 0, 4),
		(8, 4, 0, 4, 4, 0, 4),
	\\
	&&
		(8, 4, 0, 6, 0, 4, 2),
		(8, 4, 2, 4, 0, 6, 0),
		(8, 4, 4, 0, 4, 4, 0),
		(8, 6, 0, 4, 2, 4, 0),
	\\
	&&
		(6, 0, 6, 0, 6, 0, 6),
		(6, 6, 0, 6, 0, 6, 0),
		(0, 0, 4, 6, 4, 0, 10),
		(0, 0, 4, 8, 0, 4, 8),
	\\
	&&
		(0, 0, 8, 4, 0, 8, 4),
		(0, 0, 8, 0, 8, 0, 8),
		(0, 0, 10, 0, 4, 6, 4),
		(0, 4, 0, 8, 4, 0, 8),
	\\
	&&
		(0, 4, 0, 10, 0, 4, 6),
		(0, 4, 6, 4, 0, 10, 0),
		(0, 4, 8, 0, 4, 8, 0),
		(0, 6, 4, 0, 10, 0, 4),
	\\
	&&
		(0, 8, 0, 4, 8, 0, 4),
		(0, 8, 0, 8, 0, 8, 0),
		(0, 8, 4, 0, 8, 4, 0),
		(0, 10, 0, 4, 6, 4, 0).
\end{eqnarray*}
\end{example}
Provided that we proved that the associated vector $\nu_{\columnAtriangulation}$ of a  non-vertical regular triangulation $\columnAtriangulation$ as above does not give any vertex of $\hyperdiscriminantpolytopesurface$, the following conjecture would hold.
\begin{conjecture}\label{conj:main}
		Let $A$ be a point configuration $A \subset \mathbb{Z}^{2}$.
	Let $\polytope$ be its convex hull.
	The Hurwitz vectors of all regular triangulations of $(Q,A)$ are in one-to-one correspondence with the vertices of the convex hull of the vectors $\nu_{\columnAtriangulation}$ for all regular triangulations of $(\columnA, \widetilde{A})$.
	In particular, if the associated toric surface $\toricvariety$ is smooth,
	then the Hurwitz polytope of $\toricvariety$ coincides with the convex hull of the Hurwitz vectors of all regular triangulations of $(Q,A)$.
\end{conjecture}
\begin{remark}
		The second statement of Conjecture \ref{conj:main} will be proved in \cite{sano} for polarized smooth toric varieties of general dimension  by employing some results in K\"ahler geometry.
		However, the first statement of Conjecture \ref{conj:main} as a combinatorial problem is open.\end{remark}

\section{Examples}\label{sec:examples}
We will see that Conjecture \ref{conj:main} holds for some toric surfaces.
We collect such examples from \cite{sturmfels17} (with \cite{kl20}) and some reflexive polytopes.
A part of the computation in this section is carried out by SageMath with TOPCOM.

\begin{example}[Example 3.1 \cite{sturmfels17}]
Let $A$ be the point configuration given by
$$
	\{
		(0,0),(1,0),(1,1),(0,1)
	\}.
$$
We label the points by $\omega_i\,\,(1\le i \le 4)$ in order.
The associated variety $\toricvariety$ is the Segre variety $\mathbb{P}^1\times\mathbb{P}^1$ in $\mathbb{P}^3$.
The polytope $(Q,A)$ has only two regular triangulations:
$$
	T_1= \{
		(\omega_1,\omega_2,\omega_3),\,
		(\omega_1,\omega_3,\omega_4)
	\},\,\,
	T_2= \{
		(\omega_2,\omega_3,\omega_4),\,
		(\omega_1,\omega_2,\omega_4)
	\},
$$
while $(\columnA,\widetilde{A})$ has $74$ regular triangulations.
The corresponding Hurwitz vectors are given by 
$$
	\xi_{T_1}=(2,0,2,0),\,
	\xi_{T_2}=(0,2,0,2).
$$
The middle point $(1,1,1,1)$ between $\xi_{T_1}$ and $\xi_{T_2}$ are given by a regular triangulation containing the cubic mixed simplex $(\pointupper_1,\pointupper_3,\pointlower_2,\pointlower_4)$.
\end{example}

\begin{example}[Example 3.2 \cite{sturmfels17} with Example 2.3 \cite{kl20}]

Let $A$ be the point configuration given by
$$
	\{
		(0,0),(0,1),(0,2),(1,0),(1,1),(2,0)
	\}.
$$
We label the points by $\omega_i\,\,(1\le i \le 6)$ in order (we follow the order in Example 2.3 \cite{kl20}).
The associated variety $\toricvariety$ is the Veronese surface $\mathbb{P}^2$ in $\mathbb{P}^5$.
The polytope $(Q,A)$ has $14$ regular triangulations as shown in Figure 1 \cite{kl20}, while $(\columnA,\widetilde{A})$ has $28080$ regular triangulations.
The corresponding Hurwitz vectors are given by 
$$
	(4,0,1,0,6,1),\,
	(3,2,0,0,6,1),\,
	(3,0,1,2,6,0),\,
	(2,2,0,2,6,0),
$$
$$
	(1,0,4,6,0,1),\,
	(0,2,3,6,0,1),\,
	(1,0,3,6,2,0),\,
	(0,2,2,6,2,0),
$$
$$
	(1,6,1,0,0,4),\,
	(0,6,1,2,0,3),\,
	(1,6,0,0,2,3),\,
	(0,6,0,2,2,2),
$$
$$
	(4,0,4,0,0,4),\,
	(0,4,0,4,4,0)
$$
in the order of the regular triangulations described in Example 2.3 \cite{kl20}.
This computation coincides with the computation of the Hurwitz polytope of $\toricvariety$ in Example 3.2 \cite{sturmfels17}.
\end{example}

\begin{example}[Reflexive polytopes]
A polytope $Q\subset\mathbb{R}^n$ is \textit{reflexive} if its vertices are primitive lattice points and its polar dual polytope is also a lattice polytope.
The reflexive polytopes correspond uniquely to Gorenstein toric Fano varieties.
Such varieties are studied well in the context of the problem of K\"ahler-Einstein metrics.
For a polytope $Q$, we take $A$ by the set of all lattice points on $Q$.
Remark that by definition, $Q$ has only one interior lattice point.

The reflexive polytopes are classified completely in low dimensions.
We refer to the labelling of two dimensional reflexive polytopes indicated in Proposition 3.4.1 \cite{nill05}.
We confirm that Conjecture \ref{conj:main} holds for some reflexive polygons shown in Table \ref{table:reflexive}.
\begin{remark}
	Example 6a corresponding to the point configuration defined in (\ref{eq:pointconfiguration}) is the only \textit{smooth} surface among Table \ref{table:reflexive}.
	Even if $\toricvariety$ is not smooth, we still denote the convex hull of the vectors $\nu_{\columnAtriangulation}$ by $\hyperdiscriminantpolytopesurface$ in Table \ref{table:reflexive}.	
\end{remark}
\begin{table}[h]
	\caption{Reflexive Polygons}
	\label{table:reflexive}
	\begin{tabular}{|c|c|c|c|c|}
		\hline
		$Q$ & $T$ & $\columnAtriangulation$ & $\hyperdiscriminantpolytopesurface$ & normally equivalence 
		\\ \hline
		$3$ &$2$&$84$&$2$& true 
		\\ \hline
		$4a$ &$3$&$544$&$3$& true 
		\\ \hline
		$4b$ &$4$&$1270$&$4$& true 
		\\ \hline
		$4c$ &$4$&$844$&$4$& true 
		\\ \hline
		$5a$ &$10$&$26540$&$10$& true 
		\\ \hline
		$5b$ &$12$&$33380$&$12$& true 
		\\ \hline
		$6a$ & $32$ &$928930$& $32$ & true 
		\\ \hline
		$6b$ &$35$&$980824$&$35$& true 
		\\ \hline
		$6c$ &$35$&$980824$&$35$& true 
		\\ \hline
		$6d$ &$32$&$696710$&$32$& true 
		\\
		\hline
	\end{tabular}
\end{table}
The number in the first column indicates the label of the reflexive polygons in \cite{nill05}.
This number also indicates the number of lattice points on the boundary.
The number in the second (resp. third) column indicates the number of regular triangulations of $(Q,A)$ (resp. $(\columnA, \widetilde{A})$).
The number in the fourth column indicates whether $\hurpolytope=\hyperdiscriminantpolytopesurface$ and $\chowpolytope$ are normally equivalent or not.
\end{example}

\section{Relations with $K$-stabilites}
\label{sec:k-stability}
We discuss relations of the Hurwitz vectors to two kinds of $K$-stabilities: $K$-stability of pairs defined by Paul \cite{paul12} and toric $K$-stability defined by Donaldson \cite{donaldson02}.

\subsection{$K$-semistability of pairs}
First, we provide an application of Theorem \ref{thm:main_intro} to $K$-semistability of pairs by Paul \cite{paul12}. 

Let us recall its definition briefly.
Let $X$ be an $n$-dimensional, smooth, linearly normal, complex algebraic variety in $\mathbb{P}^N$ of degree $d_X \ge 2$.
Take a maximal algebraic torus $H$ in $\mathrm{SL}(N+1,\mathbb{C})$.
Let $\chowform$ be the Chow form of $X$, which is the defining polynomial of the divisor on the Grassmannian $\mathbb{G}(N-n-1, \mathbb{P}^N)$ defined by
$$
	\{ L \in \mathbb{G}(N-n-1, \mathbb{P}^N) \mid\, L\cap X\neq \emptyset\}.
$$
Let $\chowpolytopeh$ be the Chow polytope of $X$, i.e., the weight polytope of $\chowform$ with respect to $H$.
On the other hand, we denote by $\hyperdiscriminantpolytopeh$ the weight polytope of the hyperdiscriminant (Hurwitz form) of $X$ with respect to $H$.
Remark that $H$ acts on $\mathbb{P}^N\times\mathbb{P}^{n-1}$ so that it acts on the second factor trivially. 
Remark that we denote the weight polytopes with respect to the $\torus$-action by $\chowpolytope$ and $\hyperdiscriminantpolytope$  as ever.

We say that the pair $(\chowform^{\deg(\hyperdiscriminant)}, \hyperdiscriminant^{\deg(\chowform)})$ is $K$-semistable with respect to a maximal torus $H$ in $\mathrm{SL}(N+1,\mathbb{C})$ if and only if 
\begin{equation}\label{eq:ksemistable}
	\deg(\hyperdiscriminant)
	\chowpolytopeh
	\subseteq
	\deg(\chowform)
	\hyperdiscriminantpolytopeh.	
\end{equation}
Remark that both of $\deg(\hyperdiscriminant)$ and $\deg(\chowform)$ are written in Pl\"ucker coordinates on the Grassmannian.

Theorem \ref{thm:main_intro} implies a condition to $K$-semistability of pairs in the case of toric surfaces if $H$ is the standard torus.

\begin{corollary}\label{cor:ksemistability}
	Let $A$ be a point configuration $A \subset \mathbb{Z}^{2}$.
	Let $\polytope$ be its convex hull.
	Assume that the associated toric surface $\toricvariety$ is smooth.
	Let $\conv(\{\xi_T\})$ be the convex hull of the Hurwitz vectors $\xi_T$ of all regular triangulations $T$ on $(Q,A)$.
	If the scaled convex hull of $\conv(\{\xi_T\})$  dominates the scaled Chow polytope
	\begin{equation}\label{eq:convexhurwitz_chowpolytope}
		\deg(\Delta_{X_A\times \mathbb{P}^1})
		\chowpolytope
		\subseteq
		\deg(R_{X_A}) \conv(\{\xi_T\}),
	\end{equation}
	then the pair $(R_{X_A}^{\deg(\Delta_{X_A\times \mathbb{P}^1})}, \Delta_{X_A\times \mathbb{P}^1}^{\deg(R_{X_A})})$ is $K$-semistable with respect to the standard torus $H_{\mathrm{st}}\simeq (\mathbb{C}^\times)^N$ given by
$$
	\begin{pmatrix}
		t_{1} &  &  \\
		 & \ddots & \\
		 && (t_1\cdots t_{N})^{-1}
	\end{pmatrix} \in \mathrm{SL}(N+1,\mathbb{C}).
$$	
\end{corollary}
\begin{proof}
It is sufficient to see a relation with $\chowpolytopeh$ (resp. $\hyperdiscriminantpolytopeh$) and $\chowpolytope$ (resp. $\hyperdiscriminantpolytope$) (cf. \cite{yotsutani2019}).
Let $\Pi: \mathbb{R}^{N+1}\to \mathbb{R}^N$ be the projection defined by
$$
	(x_1,\ldots,x_{N-1},x_N) \mapsto
	(x_1-x_N, \ldots, x_{N-1}-x_N).
$$
By definition, we have
\begin{eqnarray*}
		\Pi(\chowpolytope)
	&=&
		\Delta_{H_{\mathrm{st}}}(\chowform),
	\\
		\Pi(\hyperdiscriminantpolytope)
	&=&
		\mathcal{W}_{H_{\mathrm{st}}}(\hyperdiscriminantpolytope).
\end{eqnarray*}
Hence, Theorem \ref{thm:main_intro} and (\ref{eq:convexhurwitz_chowpolytope}) implies (\ref{eq:ksemistable}).
The proof is completed.
\end{proof}

By definition of the Hurwitz vectors, we can expect that the convex hull $\conv(\{\xi_T\})$ would have similar combinatorial properties as the Chow polytope $\chowpolytope$.
For instance, we can prove the following directly from the definition of the Hurwitz vectors.
\begin{proposition}\label{prop:polytopes}
	Let $A$ be a point configuration $A \subset \mathbb{Z}^{2}$.
	Let $\polytope$ be its convex hull.
	Assume that the associated toric surface $\toricvariety$ is smooth.
	Let $\conv(\{\xi_T\})$ be the convex hull of the Hurwitz vectors $\xi_T$ of all regular triangulations $T$ on $(Q,A)$.
	Assume that the length of any edge of $Q$ is equal to one.
	Then the followings hold.
	\begin{enumerate}
		\item\label{obs:list-1}
			The number of the vertices of $\conv(\{\xi_T\})$ is equal to the number of the vertices of the Chow polytope $\chowpolytope$.
		\item\label{obs:list-2}
			The edges of $\conv(\{\xi_T\})$ are in one-to-one correspondence with the edges of $\chowpolytope$ such that they are parallel to one another.
			In particular, the convex hull $\conv(\{\xi_T\})$ is \textit{normally equivalent} to $\chowpolytope$, i.e., their normal fans coincide.
	\end{enumerate}
\end{proposition}
\begin{proof}
Since any edge of $Q$ has the length one, $\eta_{T,1}(\omega)$ is equal to two if $\omega$ is a vertex of $Q$ or zero if $\omega$ is an interior point of $Q$.
This implies that $\xi_{T}=\xi_{T’}$ if and only if $T=T’$, because 
\begin{equation}\label{eq:xi_and_eta}
	\xi_{T}-\xi_{T’}=\eta_{T,2}-\eta_{T’,2}.
\end{equation}
By the same proof of Theorem 1.7 in Chapter 7 \cite{gkz94}, we find that the set of the Hurwitz vectors of all regular triangulations coincides with  the set of the vertices of the convex hull $\conv(\{\xi_T\})$.
This proves the first statement.
The second statement also follows from (\ref{eq:xi_and_eta}).
The proof is completed.
\end{proof}

We confirm by using a computer that the above proposition still holds among the examples in the previous section (cf. the fourth column in Table \ref{table:reflexive}), in which $Q$ has an edge of length larger than two.
Assuming Conjecture \ref{conj:main}, the Hurwitz polytope also will have the properties in Proposition \ref{prop:polytopes}.
Such properties are pointed out by Sturmfels in Example 3.2  \cite{sturmfels17}.
In particular, he gives a counter-example which shows that the first statement in Proposition \ref{prop:polytopes} does not hold in general.
Provided that we knew what kind of varieties satisfy the properties in Proposition \ref{prop:polytopes}, it would be useful for further study of $K$-stability of pairs.

\subsection{Toric $K$-stability}
We describe the toric Non-Archimedean $K$-energy in toric $K$-stability of \cite{donaldson02} by the GKZ vectors and the Hurwitz vectors.

Let us recall the definitions.
Let $(X,L)$ be a polarized toric manifold of dimension $n$ with momentum polytope $Q\subset \mr=\mathbb{R}^n$.
For a convex, rational, piecewise-linear function $f$ on $Q$, we define a convex polytope 
$$
	\{
		(x,\lambda)\in \mr\times \mathbb{R}\mid\,
		x\in Q,\, f(x) \le \lambda \le \max f
	\}.
$$
This polytope corresponds to some toric degeneration $\pi: \overline{\mathcal{X}}\to\mathbb{P}^1$ of $X$, which is a compactification of a so-called toric test configuration, i.e., $\pi: \mathcal{X}:=\overline{\mathcal{X}}\backslash \{\pi^{-1}(\infty)\}\to \mathbb{C}$ is a $\mathbb{C}^\times$-equivariant flat family of polarized schemes where $\pi^{-1}(t)=(X,L)$ for $t\neq 0$.
For each function $f$ as above, we define 
$$
	L(f)= \int_{\partial Q} fd \nu - n\frac{\vol(\partial Q)}{\vol(Q)} \int_Q f dx.
$$
We call $L(f)$ the toric non-Archimedean $K$-energy with respect to $f$ (\cite{bhj2019, hisamoto16}).
The measure $dx$ denotes the Lebesgue measure and $d\nu$ is the measure on the boundary of $Q$ so that
$
	dx_1 \ldots dx_n = \pm d\nu\wedge dh.
$
In above, $h$ is the defining polynomial of a facet of $Q$ which is the form of
$$
	h(x)= \langle x, u \rangle +c	
$$
where $u$ is a primitive normal vector of the facet and $c$ is some constant.
Notice that 
$$
	\vol(Q) = n! \int_Q dx, \,\,
	\vol(\partial Q) = (n-1)!\int_{\partial Q} d\nu.
$$
We 
say that $(X,L)$ is toric $K$-stable if and only if $L(f)\ge 0$ for any $f$ and the equality holds only if $f$ is affine.

Now, we shall see that the toric non-Archimedean $K$-energy $L(f)$ can be written in terms of the GKZ vectors and the Hurwitz vectors.
Take a convex, rational, piecewise-linear function $f$ on $Q$.
By replacing $Q$ by $kQ$ for large integer $k$, we can assume that the vertical projection of the graph of the function $f$ provides a regular subdivision $T’$ of $(Q,A)$.
Here $A$ denotes a point configuration consisting of all lattice points on $Q$.
From \cite{drs10}, there exists a regular triangulation $T$ of $(Q,A)$, which is a refinement of $T’$.

\begin{proposition}\label{prop:donaldson-futaki}
	Let $T$ be a regular triangulation associated with a given convex piecewise-linear function $f$ on $Q$ as above. 
	Then we have
	\begin{eqnarray*}
		\nonumber
		&&
		(n+1)!
		\vol(Q)L(f)
		\\
		\nonumber
		&&
		\quad\quad
		= 
		\langle 
		f,
		n\deg(\hyperdiscriminant)\eta_{T,n}
		- 
		(n+1)\deg(\chowform)\xi_T 
		\rangle
		\\
		&&
		\quad\quad
		:= 
		\sum_{\omega\in A}
		f(\omega)\big(
		n\deg(\hyperdiscriminant)
		\eta_{T,n}(\omega)
		-
		(n+1)\deg(\chowform)\xi_T(\omega)
		\big).
	\end{eqnarray*}
\end{proposition}
\begin{proof}
The proof follows essentially from the fact written in the proof of Lemma 1.8 in Chapter 7 \cite{gkz94}: for an affine function $g$ on an $n$-dimensional (integral) simplex, 
\begin{equation}\label{eq:affine_integral}
	\int_\sigma g dx= \frac{\mathrm{Vol}(\sigma,dx)}{n+1} \sum_{\omega} g(\omega)
	= \frac{\vol(\sigma)}{(n+1)!} \sum_{\omega} g(\omega)
\end{equation}
where $\omega$ runs through all vertices of $\sigma$.
Recall that 
\begin{equation}\label{eq:deg_deg}
	\degree = \deg(\chowform)= \vol(Q),\quad
	\mu= n\frac{\vol(\partial Q)}{\vol(Q)}
\end{equation}
where $\mu$ denotes the average of the scalar curvature as before.
Then the formula (5.53) in \cite{paul12} says that
\begin{eqnarray}
	\nonumber
		\deg(\hyperdiscriminant)
	&=&
		(n+1)\deg(X)-\frac{\deg(X)}{n}\mu
	\\
	\label{eq:degree_hurwitz_general}
	&=&
		(n+1)\vol(Q)- \vol(\partial Q).
\end{eqnarray}
Remark that $\deg(\hyperdiscriminant)$ is written in the Pl\"ucker coordinates.
By (\ref{eq:affine_integral}) and the definitions of $\eta_{T,n}$ and $\xi_{T}$, we have
\begin{equation}\label{eq:integral_GKZ}
		\langle f, \eta_{T,n} \rangle
	=
		(n+1)!\int_{Q} f dx,
\end{equation}
and 
\begin{equation}\label{eq:integral_Hurwitz}
		\langle f, \xi_T \rangle
	=
		n\cdot (n+1)!\int_{Q} f dx
		-	
		n! \int_{\partial Q} fd\nu.
\end{equation}
From (\ref{eq:deg_deg}), (\ref{eq:degree_hurwitz_general}), (\ref{eq:integral_GKZ}) and (\ref{eq:integral_Hurwitz}), we have the desired equality.
The proof is completed.
\end{proof}

\end{document}